\documentclass[smallextended]{svjour3}       % onecolumn (second format)
\smartqed  % flush right qed marks, e.g. at end of proof

\usepackage{amsmath,amssymb,amsthm,amsfonts,latexsym,graphicx,subfigure}
\usepackage[numbers,sort&compress]{natbib}
\usepackage{graphicx}
\usepackage{stackrel}

\usepackage{tikz}
\usetikzlibrary{topaths,calc}
\usetikzlibrary{positioning}
\tikzset{bluenode/.style={circle,fill=gray!50,minimum size=0.4cm,inner sep=0pt},}
\tikzset{rednode/.style={circle,fill=black!100,minimum size=0.4cm,inner sep=0pt},}

\DeclareMathOperator{\aut}{Aut}
\DeclareMathOperator{\supp}{supp}
\DeclareMathOperator{\card}{card}

\theoremstyle{definition}
\newtheorem{ex}[theorem]{Example}

 \journalname{Theory in Biosciences}
\begin{document}

\title{Geometry and symmetry in biochemical reaction systems \thanks{This work was supported by The Alan Turing Institute under the EPSRC grant EP/N510129/1.}
}


\author{Raffaella Mulas  \and Rubén J. Sánchez-García \and Ben D. MacArthur}

\authorrunning{R. Mulas, R. J. Sánchez-García, B. D. MacArthur} % if too long for running head

\institute{Mathematical Sciences, University of Southampton, UK\\
Institute of Life Sciences, University of Southampton, UK\\
The Alan Turing Institute, London, UK\\
\email{R.Mulas@soton.ac.uk (Raffaella Mulas)}
}

\date{Received: date / Accepted: date}

\maketitle

\begin{abstract}
Complex systems of intracellular biochemical reactions have a central role in regulating cell identities and functions. Biochemical reaction systems are typically studied using the language and tools of graph theory. However, graph representations only describe pairwise interactions between molecular species, and so are not well suited to modelling complex sets of reactions that may involve numerous reactants and/or products. Here we make use of a recently-developed hypergraph theory of chemical reactions that naturally allows for higher-order interactions to explore the geometry and quantify functional redundancy in biochemical reactions systems. Our results constitute a general theory of automorphisms for oriented hypergraphs and describe the effect of automorphism group structure on hypergraph Laplacian spectra.
\keywords{Complex systems \and Hypergraphs \and Spectral properties \and Symmetry}
\end{abstract}

\section{Introduction}

Many real-world complex systems can be modelled as graphs in which vertices represent system elements and edges pairwise interactions between those elements \cite{newman2018networks,barabasi2016network}. This approach allows powerful tools from graph theory to be used in the analysis of complex systems in numerous domains -- from technological networks such as power grids, the internet and world-wide-web to biological networks such as food webs and molecular interaction networks, and social networks such as those that arise in online social media -- and has been tremendously successful in discerning important structural and dynamical properties of the complex systems they represent \cite{newman2003structure}. 

Notably, a number of features are common to many disparate real-world networks, and they have also been observed in classical random graph models, such as the Barabási-Albert model and the Watts-Strogatz model. Examples include the presence of highly connected `hub' vertices \cite{barabasi1999emergence}, over-representation of important sub-graphs or `motifs' \cite{milo2002network} and the presence of local clustering \cite{watts1998collective}. In recent years it has also become clear that many real-world networks also contain a large amount of structural redundancy (i.e.~duplication of structural features), which, in turn, relates to the robustness and resilience of the underlying system. 

Mathematically, the presence of structural redundancy is quantified by the graph automorphism group \cite{symm1}, which identifies structurally equivalent vertices and edges. This allows tools from group theory to be used in network analysis and has seen a number of fruitful applications most notably in studies of robustness and resilience, efficient communication, group consensus, anonymization, compression, and patterns of network collective dynamics such as synchronisation \cite{symm1,symm3,pecora2014cluster,klickstein2019symmetry,wu2010k}. 

Moreover, a powerful tool for studying the structural properties of graphs is \emph{spectral theory}. Given a graph $\Gamma$ and a square matrix associated with $\Gamma$, such as its adjacency matrix $A$, its Kirchhoff Laplacian $\Delta$ or its normalised Laplacian $L$, the \emph{spectrum} of each of these operators, i.e.\ the multiset of its eigenvalues, is known to encode many important qualitative properties of $\Gamma$ \cite{Chung,Brouwer}. Spectral theory studies the properties that are encoded by the spectra of these operators and it is common both in pure mathematics and in applied sciences. Notably, for \emph{regular} graphs, i.e.\ graphs in which all vertices have the same degree, the spectral properties of $A$, $\Delta$ and $L$ are equivalent, as their eigenvalues only differ by an additive or a multiplicative constant in this case. For general graphs, the spectral properties of the three matrices may be slightly different, although they are typically strongly related. Also, since $A$ has both positive and negative eigenvalues while $\Delta$ and $L$ have non-negative eigenvalues, studying the spectral properties of the Laplacian matrices is often easier. Moreover, the eigenvalues of $L$ are normalised with respect to the eigenvalues of $\Delta$ and they are related to random walks on graphs, therefore studying spectral theory from the point of view of the normalised Laplacian is often preferred -- and is the approach we will take here. We refer to \cite{Chung,Brouwer} for two classical monographs on this subject. 

However, graph theory-based analyses necessarily only consider system elements and their \emph{pairwise} interactions. In many cases, higher-order interactions are also important and can play a significant part in system function \cite{carlsson2009topology}. There is increasing interest in accounting for such higher-order structures, for example by encoding system structures either as simplicial complexes, which can be analysed using tools from algebraic topology, or, more generally, as hypergraphs. Both approaches have proven successful and are active areas of current research \cite{zomorodian2005topology,Hypergraphs,Cellularnetworks,horak2013}. 

The role of higher-order interactions is particularly important when considering systems of chemical reactions. For example, proteins typically perform their functions in cells by interacting physically to form chemical complexes. While protein-protein interaction networks enumerate possible pairwise interactions, they are not able to unambiguously capture the formation of higher order complexes involving three or more proteins. More generally, biochemical reactions typically involve more than two reactants and/or products. Thus, complex systems of biochemical reactions are not well described using the language of graph theory. Yet, they can be well modelled using hypergraphs which allow \emph{hyper}edges involving more than two vertices. 

Here we develop a general theory of automorphisms for \emph{oriented hypergraphs}: a generalisation of classical hypergraphs with the additional structure that each vertex in a hyperedge is either an input or an output. Oriented hypergraphs were introduced in \cite{shi1992} and, as shown in \cite{Hypergraphs}, they are a useful tool for the modelling of chemical reaction networks. The adjacency matrix and the Kirchhoff Laplacian for oriented hypergraphs were introduced in \cite{ReffRusnak}, as a generalization of the classical ones for graphs. Moreover, the normalised Laplacian for oriented hypergraphs was introduced in \cite{Hypergraphs}. The spectral properties of these operators, as well as possible applications, have been widely studied, see for instance \cite{Hypergraphs,Master-Stability,Sharp,MulasZhang,AndreottiMulas,spectralclasses,hyp2014,hyp2015,hyp2019,chen2018,ReffRusnak}, yet a general framework to study oriented hypergraph automorphisms is still lacking. As in the graph case, the spectral properties of these three operators are similar; the adjacency matrix has both positive and negative eigenvalues while the Laplacian matrices have non-negative eigenvalues, and the spectrum of $L$ is normalised with respect to the spectrum of $\Delta$. For this reason, we will focus on spectral properties of the normalised Laplacian matrix here. 

The paper is structured as follows. In Section \ref{section:prel} we provide an overview of some required definitions related to oriented hypergraphs. In Section \ref{section:special} we show how the classical theory of graph automorphisms can be extended to hypergraphs, and outline some key differences between graph and hypergraph automorphisms. In Section \ref{section:signed} we propose a further extension of this theory that takes hyperedge signs into account. We conclude with a discussion of the relevance of this general theory to systems of biochemical reactions.

\section{Preliminary definitions}\label{section:prel}
We start by introducing some preliminary definitions. We keep the set of definitions limited to those strictly needed for the new results in later sections.

\begin{definition}[\cite{shi1992}]
An \emph{oriented hypergraph} is a pair $\Gamma=(V,H)$ where $V$ is a finite set of \emph{vertices} and $H$ is a set such that every element $h \in H$ is a pair of disjoint subsets of vertices $h=(h_{in},h_{out})$ (input and output), that is, $h_{in}, h_{out} \in \mathcal{P}(V)$, where we write $\mathcal{P}(V)$ for the power set of $V$. The elements of $H$ are called the \emph{oriented hyperedges} (or, simply, \emph{hyperedges}). Changing the orientation of a hyperedge $h$ means exchanging its input and output, leading to the pair $(h_{out},h_{in})$. The \emph{vertices of a hyperedge} $h=(h_{in},h_{out})$ are the elements of $h_{in} \cup h_{out} \subseteq V$. Two vertices in $i,j \in h$ are called \emph{co-oriented} if $i,j \in h_{in}$ or $i, j \in h_{out}$, and \emph{anti-oriented} otherwise.
\end{definition}

A classical hypergraph can be seen as an oriented hypergraph if one forgets about the input-output structure. In this sense, oriented hypergraphs generalise the standard notion of hypergraphs \cite{bretto2013hypergraph}. 
To illustrate these ideas, Figure \ref{figure:toy} shows an oriented hypergraph with five vertices and two hyperedges.
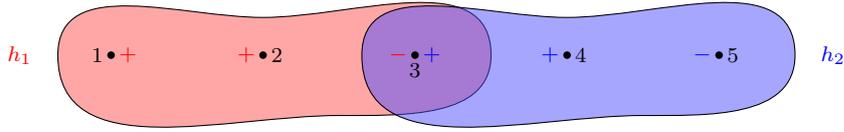
\begin{figure}[t!]
	\begin{center}
		\begin{tikzpicture}
		\node (v1) at (1,0) {};
		\node (v2) at (3,0) {};
		\node (v3) at (5,0) {};
		\node (v4) at (7,0) {};
		\node (v5) at (9,0) {};
		
		\begin{scope}[fill opacity=0.5]
		\filldraw[fill=red!70] ($(v1)+(-0.7,0)$) 
		to[out=90,in=180] ($(v2) + (0,0.5)$) 
		to[out=0,in=90] ($(v3) + (1,0)$)
		to[out=270,in=0] ($(v2) + (1,-0.8)$)
		to[out=180,in=270] ($(v1)+(-0.7,0)$);
		\filldraw[fill=blue!70] ($(v3)+(-0.7,0)$) 
		to[out=90,in=180] ($(v4) + (0,0.5)$) 
		to[out=0,in=90] ($(v5) + (1,0)$)
		to[out=270,in=0] ($(v4) + (1,-0.8)$)
		to[out=180,in=270] ($(v3)+(-0.7,0)$);
		\end{scope}

		\fill (v1) circle (0.05) node [left] {$1$} node [right] {\color{red}$+$};
		\fill (v2) circle (0.05) node [right] {$2$} node [left] {\color{red}$+$};
		\fill (v3) circle (0.05) node [below] {$3$} node [right] {\color{blue}$+$} node [left] {\color{red}$-$};
		\fill (v4) circle (0.05) node [right] {$4$} node [left] {\color{blue}$+$};
		\fill (v5) circle (0.05) node [right] {$5$} node [left] {\color{blue}$-$};

		\node at (-0.2,0) {\color{red} $h_1$};
		\node at (10.5,0) {\color{blue} $h_2$};
		\end{tikzpicture}
	\end{center}
	\caption{\emph{Example hypergraph}. An oriented hypergraph with five vertices $1$ to $5$ and two hyperedges $h_1$ and $h_2$. The hyperedge $h_1$ has $1$ and $2$ as inputs and $3$ as output; the hyperedge $h_2$ has $3$ and $4$ as inputs and $5$ as output.}
	\label{figure:toy}
\end{figure}

Oriented hypergraphs offer a valid model for biochemical networks \cite{Hypergraphs}. Each vertex may be thought of as a chemical substance and each hyperedge as a chemical reaction involving the substances that it contains as vertices (i.e.\ reactants and/or products of the reaction). The input-output structure then represents the reactant-product structure of chemical reactions. 

\begin{definition}[\cite{ReffRusnak}]The \emph{degree} of a vertex $i$, denoted $\deg(i)$, is the number of hyperedges containing $i$. The \emph{cardinality} of a hyperedge $h$, denoted $\card(h)$, is the number of vertices in $h$.
\end{definition}

For the rest of this article, let us fix such an oriented hypergraph $\Gamma=(V,H)$ on $n$ vertices labelled $1,2,\ldots,n$ (that is, we assume $V=\{1,2,\ldots,n\}$) and $m$ hyperedges $h_1,\ldots, h_m$. We also assume that $\Gamma$ has no vertices of degree zero, that is, every vertex belongs to at least one hyperedge. We define the following matrices associated with $\Gamma$.

\begin{definition}[\cite{Hypergraphs}] \label{def:incidencematrix}
The $n\times m$ \emph{incidence matrix} of $\Gamma$ is $\mathcal{I}=\mathcal{I}(\Gamma)=(\mathcal{I}_{ih})_{i\in V, h\in H}$, where
	\begin{equation*} 
	\mathcal{I}_{ih}:=\begin{cases} 1 & \text{ if }i\in h_{in}\\ -1 & \text{ if }i\in h_{out}\\ 0 & \text{otherwise.} \end{cases}
	\end{equation*}
We call $\mathcal{I}_{ih}$ the \emph{sign} of vertex $i$ in hyperedge $h$, and use the `$+$' or `$-$' symbols to represent non-zero signs in a graphical representation of a hypergraph (e.g.~Fig.~\ref{figure:toy}). 
\end{definition}

\begin{definition}[\cite{ReffRusnak}]
The $n\times n$ diagonal \emph{degree matrix} of $\Gamma$ is $D=D(\Gamma)=(D_{ij})$, where
    \begin{equation*}
    D_{ij}:=\begin{cases} \deg (i) & \text{if }i=j\\ 0 & \text{otherwise}. \end{cases} 
\end{equation*}
    \end{definition}
    
Given vertices $i, j \in V$, let us write $\deg^+(i,j)$ for the number of hyperedges in which $i$ and $j$ are co-oriented, and $\deg^-(i,j)$ for the number of hyperedges in which $i$ and $j$ are anti-oriented. Note that $\deg(i)=\deg^+(i,i)$, $\deg^-(i,i)=0$, and they are both symmetric functions: $\deg^\pm(i,j)=\deg^\pm(j,i)$ for all $i, j \in V$.

\begin{definition}[\cite{ReffRusnak}]\label{def:adjacencymatrix}
The $n\times n$ \emph{adjacency matrix} of $\Gamma$ is the symmetric matrix $A=A(\Gamma)=(A_{ij})$, where $A_{ii}=0$ for all $i$ and
\begin{equation}\label{eqn:adjmatrix}
        A_{ij}:= \deg^-(i,j)-\deg^+(i,j)
\end{equation}
for all $i\neq j$. 
\end{definition}

\begin{definition}[\cite{ReffRusnak}]\label{def:unnormalisedlaplacianmatrix}
The $n\times n$ \emph{Kirchhoff Laplacian matrix} of $\Gamma$ is $\Delta=\Delta(\Gamma)=(\Delta_{ij})$, where $\Delta=D-A$. That is, 
\begin{equation}\label{eqn:unnormalisedlaplacian}
        \Delta_{ij}:= \deg^+(i,j)-\deg^-(i,j)
\end{equation}
for all $i, j$.
\end{definition}

\begin{definition}\label{def:normalisedlaplacianmatrix}
The $n\times n$ \emph{normalised Laplacian matrix} of $\Gamma$, $L=L(\Gamma)=(L_{ij})$, is $L=D^{-1}\Delta= I - D^{-1}A $, where $I$ is the $n \times n$ identity matrix. (Note that $D$ is invertible as we have removed all vertices of degree 0.) The entries of $L$ are
\begin{equation}\label{eqn:normalisedlaplacian}
        L_{ij}:= \frac{\deg^+(i,j)-\deg^-(i,j)}{\deg(i)}
\end{equation}
for all $i, j$.
\end{definition}

The Kirchhoff Laplacian matrix $\Delta$ is symmetric but the normalised Laplacian $L$ is not. However, $L$ is \emph{isospectral} (meaning that it has the same eigenvalues, counted with multiplicity) to the symmetric matrix
\begin{equation}\label{eq:ChungLaplacian}
    \mathcal{L}:=D^{1/2}L D^{-1/2}
\end{equation}
(see e.g.~\cite[Remark 2.14]{MulasZhang}) and thus has real eigenvalues. Note that the incidence matrix $\mathcal{I}$ uniquely determines the hypergraph, but, unlike graphs, this is not true for the adjacency or Laplacian matrices: two distinct hypergraphs may have the same adjacency, or Laplacian, matrix. To see this, consider the following simple example:

\begin{ex}
Let $\Gamma=(V,H)$ and $\Gamma'=(V,H')$ be two hypergraphs with vertex set $V=\{1,2,3\}$ and hyperedge sets $H=\{h_1,h_2\}$ and $H'=\{h'_1,h'_2\}$, where
\begin{itemize}
    \item $h_1$ has $1$ as input and $2$ as output, and $h_2$ has $1$ and $2$ as inputs and $3$ as output;
    \item $h'_1$ has $1$ as input and $3$ as output, and $h'_2$ has $2$ as input and $3$ as output.
\end{itemize}
These two hypergraphs are distinct ($\Gamma'$ is a graph but $\Gamma$ is not), but have the same adjacency matrix,
\begin{equation*}
    A=\begin{pmatrix}
        0 & 0 & 1\\
        0 & 0 & 1\\
        1 & 1 & 0
\end{pmatrix}.
\end{equation*}
(The cancellation $\deg^-(1,2)-\deg^+(1,2)=1-1=0$ for $\Gamma$ is undetected by this matrix.) 
\end{ex}

The terminology and matrices introduced so far generalise the similar concepts in graph theory. A simple graph $G=(V,E)$ with a choice of edge orientations is the same as an oriented hypergraph $\Gamma=(V,H)$ with $|h_{in}|=|h_{out}|=1$ for all $h=(h_{in},h_{out})\in H$. In this case, the degree of a vertex in $\Gamma$ is the same as in $G$, $d^+(i,j)=0$ for all $i \neq j$, and $d^-(i,j)=1$ if $i$ and $j$ are connected by an edge, and 0 otherwise. In particular, the degree, adjacency and Laplacian matrices for $\Gamma$ coincide with the usual definitions from graph theory for $G$.

Collectively, these results therefore indicate that properties of hypergraphs may be encoded in matrix representations that have some similarities to those of graphs, as well as some important differences. In the following sections we will outline how structural hypergraph properties -- in particular, those related to redundancy -- manifest the in hypergraph spectra. In order to motivate these general results we first introduce some established results concerning the effect of various simple structural features of an oriented hypergraph on the spectrum of its hypergraph normalised Laplacian \cite{MulasZhang,AndreottiMulas}. Some additional definitions, below, are needed to understand these features.

\begin{definition}
The \emph{auxiliary graph} of $\Gamma$, written $G(\Gamma)$, as the graph with adjacency matrix $A(\Gamma)$. This is an undirected, weighted graph with the same vertex set as $\Gamma$ and an edge between $i$ and $j$ weighted by $A_{ij} \neq 0$, and no such edge if $A_{ij}=0$.
\end{definition}

\begin{definition}[\cite{MulasZhang}]\label{def:duplicate}
Two distinct vertices $i$ and $j$ are \emph{duplicate} if the corresponding rows (equivalently, columns) of the adjacency matrix are the same, that is, 
if $A_{ik}=A_{jk}$ (or, equivalently, $A_{ki}=A_{kj}$) for all $k\in V$. In particular, $A_{ij}=A_{ji}=A_{ii}=A_{jj}=0$.
\end{definition}
\begin{definition}[\cite{AndreottiMulas}]\label{def:twin}
Two distinct vertices $i$ and $j$ are \emph{twin} if they belong to exactly the same set of hyperedges, with the same orientations, that is, 
\[
    i \in h_{in} \iff j \in h_{in} \ \text{ and }\ 
    i \in h_{out} \iff j \in h_{out},
\]
for all $h=(h_{in},h_{out})\in H$. 
\end{definition}
Note that if $i$ and $j$ are twin then $\deg^\pm(i,k)=\deg^\pm(j,k)$, and hence $A_{ik}=A_{jk}$, for all $k \in V\setminus\{i,j\}$. Moreover, $A_{ij}=A_{ji}=-\deg(i)=-\deg(j)\neq 0$ (we assume that there are no vertices of degree zero). Therefore, twin vertices cannot be duplicate vertices and vice versa.

Recall that, in oriented hypergraphs, every vertex has a sign for each hyperedge in which it is contained (Definition \ref{def:incidencematrix}). By reversing signs, we can define anti-duplicate and anti-twin vertices, as follows. 

\begin{definition}
Two vertices $i$ and $j$ are \emph{anti-duplicate} if the corresponding rows (equivalently, columns) of the adjacency matrix have opposite sign, that is, 
if $A_{ik}=-A_{jk}$ (or, equivalently, $A_{ki}=-A_{kj}$) for all $k\in V$. In particular, $A_{ij}=A_{ji}=A_{ii}=A_{jj}=0$.
\end{definition}

\begin{definition}
Two vertices $i$ and $j$ are \emph{anti-twin} if they belong exactly to the same set of hyperedges, with reversed orientations, that is, 
\[
    i \in h_{in} \iff j \in h_{out} \ \text{ and }\ 
    i \in h_{out} \iff j \in h_{in},
\]
for all $h=(h_{in},h_{out})\in H$. 
\end{definition}
Note that if $i$ and $j$ are anti-twin then $\deg^\pm(i,k)=\deg^\mp(j,k)$, and hence $A_{ik}=-A_{jk}$, for all $k \in V\setminus\{i,j\}$. Moreover, $A_{ij}=A_{ji}=\deg(i)=\deg(j)$. Therefore, anti-twin vertices cannot be anti-duplicate vertices and vice versa.

In \cite{MulasZhang} it is shown that a hypergraph that possesses $k$ duplicate vertices will have normalised Laplacian eigenvalue $1$ with multiplicity at least $k-1$. Similarly, in \cite{AndreottiMulas} it is shown that the presence of $k$ twin vertices produce the normalised Laplacian eigenvalue $0$ with multiplicity at least $k-1$. It is clear from these elementary results that structural repetition in a hypergraph naturally gives rise to repeated eigenvalues, yet the generality of these results is unclear. In the next section we interpret these results as part of a more general theory that relates structural redundancy (measured by the presence of hypergraph automorphisms) to hypergraph spectra.

\section{Redundancy and symmetry in hypergraphs}
\label{section:special}
 Informally, redundancy results in duplication of hypergraph structural features (such as vertices, hyperedges or collections of vertices and hyperedges). Moreover, from the results above it is expected that such repetition may leave a signature in its eigenvalue spectra. In this section, we show that these results are specific instances that arise from a general theory of hypergraph automorphisms, adapting the work in \cite{symm1,symm2,symm3} for hypergraphs and considering the normalised Laplacian.

\subsection{Hypergraph automorphisms} \label{section:automorphisms}
Informally, a hypergraph symmetry is a permutation of the vertices that preserves the hypergraph structure. More precisely, 
\begin{definition}
A \emph{hypergraph automorphism} is a permutation $p$ of the vertices of $\Gamma$ that preserves hyperedges, that is,
    \begin{equation*}
    p(h)=(p(h_{in}),p(h_{out})) \in H \quad \text{for all } h=(h_{in},h_{out}) \in H.
\end{equation*}
(We write $p(S)=\{p(s_1),\ldots,p(s_k)\}$ whenever $S =\{s_1,\ldots,s_k\} \subseteq V$.)
\end{definition}
Note that, since $p$ is invertible, it also induces a permutation on the hyperedges of $\Gamma$, $h \mapsto p(h)$. Moreover, hypergraph automorphisms induce automorphisms of the adjacency and Laplacian matrices, as follows. 

\begin{definition}\label{def:adj-automorphism}
An \emph{adjacency automorphism} is a permutation $p$ of the vertices of a hypergraph that preserves adjacency, that is, $A_{p(i)p(j)} = A_{ij}$ for all $1 \le i, j \le n$, where $A=A(\Gamma)$. We can write this in matrix form as
\begin{equation}\label{eqn:adjperm2}
     AP=PA,
\end{equation} 
where $P=(P_{ij})$ is the permutation matrix representing $p$, that is, $P_{ij}=1$ if $p(i)=j$, and 0 otherwise.  
\end{definition} 

\begin{definition} A \emph{Laplacian automorphism} is an adjacency automorphism $p$ that also preserves degrees, that is, $\deg (i)=\deg (p(i))$, for all $i=1,\ldots,n$. 
\end{definition}
Note that if $p$ is a Laplacian automorphism and $P$ is the permutation matrix representing $p$, then $\Delta P=P\Delta$ and $LP=PL$, that is, $p$ preserves both the Kirchhoff Laplacian and the normalised Laplacian. In general we the following inclusions hold.
\begin{proposition}
Every hypergraph automorphism is a Laplacian automorphism, and every Laplacian automorphism is an adjacency automorphism. The reciprocals of these statements hold if $\Gamma$ is a simple graph, but not in general.
\end{proposition}
\noindent{}Schematically, for graphs:
\begin{align*}
    \{\text{adjacency automorphisms}\} &=
    \{\text{Laplacian automorphisms}\}, \\ &= 
    \{\text{graph automorphisms}\},
\end{align*}
while for hypergraphs:
\begin{align*}
    \{\text{adjacency automorphisms}\} &\supseteq \{\text{Laplacian automorphisms}\}, \\ &\supseteq \{\text{hypergraph automorphisms}\}.
\end{align*}

\begin{proof}If $p$ is a hypergraph automorphism, then clearly $\deg^\pm(i,j)=\deg^\pm\left(p(i),p(j)\right)$ for all $i, j \in V$ and, in particular, $\deg(i)=\deg^+(i,i)=\deg(p(i))$. From Eq.~\eqref{eqn:adjmatrix} it is clear that $p$ is a Laplacian automorphism. Moreover, by definition, it is clear that any Laplacian automorphism is also an adjacency automorphism. The case when $\Gamma$ is a simple graph is well-known and straightforward.\newline
To see that the reciprocals do not necessarily hold in general, consider the following example. Let $\Gamma=(V,H)$ with vertex set $V=\{1,2,3\}$ and hyperedge set $H=\{h_1,h_2,h_3\}$, where
\begin{itemize}
    \item $h_1$ has $1$ as input and $2$ as output;
    \item $h_2$ has $2$ as input and $1$ as output;
    \item $h_3$ only contains the vertex $3$, as input.
\end{itemize}Then, the adjacency matrix of $\Gamma$ is the $3\times 3$ zero matrix, implying that any permutation of the vertices is an adjacency automorphism. However, the permutation $p$ such that $p(1)=3$ and $p(3)=1$ is not a Laplacian automorphism, since $\deg(1)=2\neq \deg(3)=1$.
\end{proof}

We begin by describing duplicate and twin vertices in terms of automorphisms.
\begin{proposition}\label{prop:duplicate-twin}
Let $\Gamma$ be an oriented hypergraph.
\begin{itemize}
    \item[(i)] If two vertices $i,j \in V$ are duplicate then the transposition $p=(i\ j)$ is an adjacency automorphism.
    \item[(ii)] If two vertices $i,j \in V$ are duplicate and $\deg(i)=\deg(j)$, then the transposition $p=(i\ j)$ is a Laplacian automorphism.
    \item[(iii)] If two vertices $i,j \in V$ are twin then the transposition $p=(i\ j)$ is a hypergraph automorphism.  
\end{itemize}
The converses of these statements are not necessarily true.
\end{proposition}
(For anti-duplicate and anti-twin vertices, see Proposition \ref{prop:anti}.)
\begin{proof}
(i) Let $A=A(\Gamma)$ and let $P$ be the permutation matrix of the transposition $p=(i\ j)$ (see Definition \ref{def:adj-automorphism}). Clearly, $AP$ is the matrix $A$ with the $i$th and $j$th rows swapped, and $PA$ is the matrix $A$ with the $i$th and $j$th columns swapped. By Definition \ref{def:duplicate}, the $i$th row, respectively column, of $A$ equals the $j$th row, respectively column, of $A$. In particular, $AP=PA$ and $p$ is an adjacency automorphism. The converse is not true: $AP=PA$ if and only if the $i$th row (and column) of A equals the $j$th row (and column) of $A$, except possibly $A_{ii} = A_{jj} \neq A_{ij} = A_{ji}$. In that situation, $p=(i\ j)$ is an adjacency automorphism but $i$ and $j$ are not duplicate. \newline
(ii) This point follows easily from (i) and from the definition of Laplacian automorphism. The converse is not true: assume that $\deg(i)=\deg(j)$ and the $i$th row (and column) of $A$ equals the $j$th row (and column) of $A$, except $A_{ii} = A_{jj} \neq A_{ij} = A_{ji}$. In that case, $p=(i\ j)$ is a Laplacian automorphism but $i$ and $j$ are not duplicate. \newline
(iii) If $i$ and $j$ are twin and $h \in H$, then $i, j \in h_{in}$, or $i, j \in h_{out}$, or neither $i$ nor $j$ are vertices in $h$. In all cases, $p(h)=h$, that is, $p$ acts trivially on hyperedges. In particular, $p(h)\in H$ for all $h\in H$ and $p$ is a hypergraph automorphism. The converse is not true: it is easy to find a hypergraph automorphism of the form $p=(i\ j)$ not acting as trivially on hyperedges.
\end{proof}

Now we have formalised the concept of symmetry, or redundancy, in hypergraphs (as hypergraph automorphisms), we can deduce some structural and spectral results: namely, the effects of the presence of symmetry on hypergraph spectra. 

\subsection{Structural Results}
In this section, we discuss the effects of the presence of automorphisms, as defined above, on the hypergraph structure. To begin we note that the set of Laplacian automorphisms together with the composition of permutations forms a group, denoted $\aut(\Gamma)$. Next, we explain a decomposition of $\aut(\Gamma)$ into permutations with disjoint supports. 
\begin{definition}
    Given a permutation of the vertices $p$, its \emph{support} is
    \begin{equation*}
        \supp(p):=\{i \in V \mid p(i)\neq i\}.
    \end{equation*}Two permutations are \emph{disjoint} if their supports are non-intersecting.
\end{definition}

Following \cite{symm1,symm2}, we decompose $\aut(\Gamma)$ into a direct product of subgroups that naturally reflect structural redundancy in $\Gamma$. Let $S$ be a set of generators of $\aut(\Gamma)$ not containing the identity, and let $S=S_1\sqcup\ldots\sqcup S_l$ be the (unique) irreducible partition of $S$ into support-disjoint subsets. Let $\mathcal{P}_j$ be the subgroup generated by $S_j$. Then,
\begin{equation}\label{eq:geometricdecomposition}
    \aut(\Gamma)=\mathcal{P}_1\times\ldots\times \mathcal{P}_l
\end{equation}is the unique, irreducible direct product decomposition of $\aut(\Gamma)$ (a proof follows that of \cite[Equation 1]{symm1}, we omit details here). Since it relates to hypergraph symmetry we will call \eqref{eq:geometricdecomposition} the \emph{symmetric decomposition of $\aut(\Gamma)$}. Similarly, for each $j=1,\ldots,l$ we denote 
\begin{equation*}
    M_j:=\bigcup_{\tau\in S_j}\supp(\tau).
\end{equation*}
Using this notation, we call 
\begin{equation*}
    V:=V_0\sqcup M_1\sqcup\ldots \sqcup M_l
\end{equation*} 
the \emph{symmetric decomposition of $\Gamma$}, where $V_0$ is the set of fixed points, that is, 
\[
    V_0 = \{ i \in V \mid p(i)=i \text{ for all } p \in \aut(\Gamma)\}. 
\]
As with any action of a group on a set, we have the concept of a group \emph{orbit}.
\begin{definition}
    The \emph{orbit} of $i\in V$ is
    \begin{equation*}
        \mathcal{O}(i):=\{p(i):p\in\aut(\Gamma)\}.
    \end{equation*}
\end{definition}

From this definition, a natural measure of redundancy is:
    \begin{equation*}
        r=\frac{\#\mathcal{O}-1}{n},
    \end{equation*}
    where $\#\mathcal{O}$ is the number of orbits, and $n$ the number of vertices, of $\Gamma$. Note that $1\leq \#\mathcal{O}\leq n$, so $$0\leq r\leq \frac{n}{n-1}<1.$$ In particular, $r=0$ if and only if $\#\mathcal{O}=1$, that is, all vertices (reactants in a chemical reaction system) are structurally equivalent. On the other hand, $r=\frac{n}{n-1}$ if and only if $\#\mathcal{O}=n$, that is, if and only if $\aut(\Gamma)$ is trivial and therefore there is no structural redundancy in $\Gamma$. %Indeed, if $r>0$ (equivalently, $\#\mathcal{O} > 1$) and $n$ is large, then $1/r \frac{n}{\#\mathcal{O}}$, the average orbit size. 
    
    Thus, $r$ quantifies the extent to which the oriented hypergraph $\Gamma$ is constructed from repetition of structurally equivalent units. Due to the evolutionary processes that form them, biochemical reaction systems often contain duplicated elements \cite{vazquez2003}, which gives rise to local symmetries (i.e.\ permutations of nodes that are close in the hypergraph that preserve adjacency). In the absence of global symmetries (i.e.\ permutations of nodes that are distant in the hypergraph that preserve adjacency), $r$ is then a natural measure of structural redundancy: biochemical systems with a high redundancy are robust in the sense that damage or deletion of redundant vertices or units (i.e.\ individual chemical reactants, or small sub-systems of chemical reactions) do not cause catastrophic system failures, but rather can be absorbed by their replacements and so allow the system to continue to function normally. 

\subsection{Spectral Results}

Recall that the spectrum of a matrix is the multiset of its eigenvalues. Given $\Gamma$, we define the \emph{adjacency spectrum} of $\Gamma$ as the spectrum of $A(\Gamma)$, the \emph{Kirchhoff Laplacian spectrum} as the spectrum of $\Delta(\Gamma)$ and the \emph{normalised Laplacian spectrum} as the spectrum of $L(\Gamma)$. Each of these matrices has $n$ real eigenvalues and the corresponding eigenvectors are elements of $\mathbb{R}^n$, where $n$ is the number of vertices. We will see each eigenvector as a function $f:V\rightarrow \mathbb{R}^n$ and we will therefore call them \emph{eigenfunctions}. We will focus on the spectrum of the normalised Laplacian $L$ (or, equivalently, on the spectrum of the matrix $\mathcal{L}$ defined in \eqref{eq:ChungLaplacian}). 

We may factor out any redundancy to obtain the essential structural characteristics of a reaction system $\Gamma$. In particular, given a partition of the vertex set $V=V_1\sqcup\ldots\sqcup V_l$, we define: 

\begin{definition} The \emph{quotient matrix} of $\mathcal{L}$ is $Q(\mathcal{L}):=(Q_{\alpha\beta})_{\alpha\beta}$, where\begin{equation*}Q_{\alpha\beta}:=\frac{1}{|V_\alpha|}\cdot \sum_{i\in V_\alpha, j\in V_\beta }\mathcal{L}_{ij}.\end{equation*}
\end{definition}

Note that the quotient matrix can be also written in alternative form as follows. Let $K:=\textrm{diag}(|V_1|,\ldots,|V_l|)$ and let $S$ be the $n\times l$ characteristic matrix of the partition, that is, each column $K_j$ is the characteristic vector of the set $V_j$. Then,
\begin{equation*}
    Q(\mathcal{L})=K^{-1}S^\top \mathcal{L}S.
\end{equation*}

Because $Q(\mathcal{L})$ is not necessarily symmetric, it is not immediately clear if it has real spectrum. In fact, it does have real spectrum, as can be seen from the following definition. 

\begin{definition}
Given a partition of the vertex set $V=V_1\sqcup\ldots\sqcup V_l$, the \emph{symmetric quotient matrix} of $\mathcal{L}$ is the $l\times l$ symmetric matrix $Q^{\textrm{sym}}(\mathcal{L})$ with entries
\begin{equation*}Q^{\textrm{sym}}_{\alpha\beta}:=\frac{1}{\sqrt{|V_\alpha|\cdot|V_\beta|}}\cdot \sum_{i\in V_\alpha, j\in V_\beta }\mathcal{L}_{ij}.\end{equation*}
\end{definition}

Note that the symmetric quotient matrix of $\mathcal{L}$ can be written as
\begin{equation*}
    Q^{\textrm{sym}}=K^{-1/2}S^\top \mathcal{L}SK^{-1/2}=K^{1/2}QK^{-1/2}.
\end{equation*}
Hence, $Q^{\textrm{sym}}$ and $Q$ are similar, which implies that they are isospectral and thus $Q(\mathcal{L})$ has real spectrum. Moreover, $f$ is an eigenfunction with eigenvalue $\lambda$ for $Q^{\textrm{sym}}$ if and only if $K^{-1/2}f$ is eigenfunction of $\lambda$ for $Q$.

From here on, we shall always refer to the quotient matrix and to the symmetric quotient matrix of $\mathcal{L}$ with respect to the partition of $V$ into orbits. This partition is clearly \emph{equitable} \cite{Brouwer}, i.e.\ the row sum of each block of $\mathcal{L}$ with respect to the partition is constant.

With this notation we are now in a position to consider the spectrum of $L$ in terms of its underlying automorphism group, and therefore to dissect the effect of redundancy on its spectral properties. The following result is fundamental.

\begin{proposition}\label{prop:spectrum}
The spectrum of $L$ consists of the spectrum of $Q^{\textrm{sym}}(\mathcal{L})$ (with eigenfunctions that are constant on each orbit) together with the eigenvalues belonging to eigenfunctions that sum to zero on each orbit.
\end{proposition}
\begin{proof}Use the following facts:
\begin{itemize}
    \item By \cite[Lemma 2.3.1]{Brouwer}, the spectrum of $\mathcal{L}$ consists of the spectrum of $Q(\mathcal{L})$ (with eigenfunctions that are constant on each part of the partition) together with the eigenvalues belonging to eigenfunctions that sum to zero on each part of the partition.
    \item By the considerations above, $Q(\mathcal{L})$ is isospectral to  $Q^{\textrm{sym}}(\mathcal{L})$.
       \item By \cite[Remark 2.14]{MulasZhang}, $L$ is isospectral to $\mathcal{L}$ and $f$ is an eigenfunction with eigenvalue $\lambda$ for $\mathcal{L}$ if and only if $D^{1/2}f$ is eigenfunction of $\lambda$ for $L$.
       \item If $f$ is either constant in the parts of the partition, or it sums to zero on each part of the partition, then the same holds for $D^{1/2}f$, since the vertices belonging to the same set of the partition have the same degree. \qedhere
\end{itemize}
\end{proof}

This result indicates that the spectrum of $\Gamma$ can be split into pieces relating to redundant and unique structural features. To deconstruct this decomposition further, the following definition is useful:
\begin{definition}
The \emph{quotient network of} $\Gamma$, denoted $Q(\Gamma)$, is the (unique) weighted, undirected graph with self-loops that has adjacency matrix $Q^{\textrm{sym}}(\mathcal{L})$.
\end{definition}
Using this definition, we can rewrite Proposition \ref{prop:spectrum} as follows.
\begin{corollary}\label{cor:quotient}
The spectrum of $\Gamma$ consists of the adjacency spectrum of its quotient network (with eigenfunctions that are constant on each orbit) together with the eigenvalues belonging to eigenfunctions that sum to zero on each orbit.
\end{corollary}
\begin{proof}It follows from Proposition \ref{prop:spectrum}, together with the fact that the adjacency matrix of $Q(\Gamma)$ is $Q^{\textrm{sym}}(\mathcal{L})$.
\end{proof}

To illustrate these ideas it is useful to consider an example. 

\begin{figure}[t!]
\centering
\includegraphics[width=.8\linewidth]{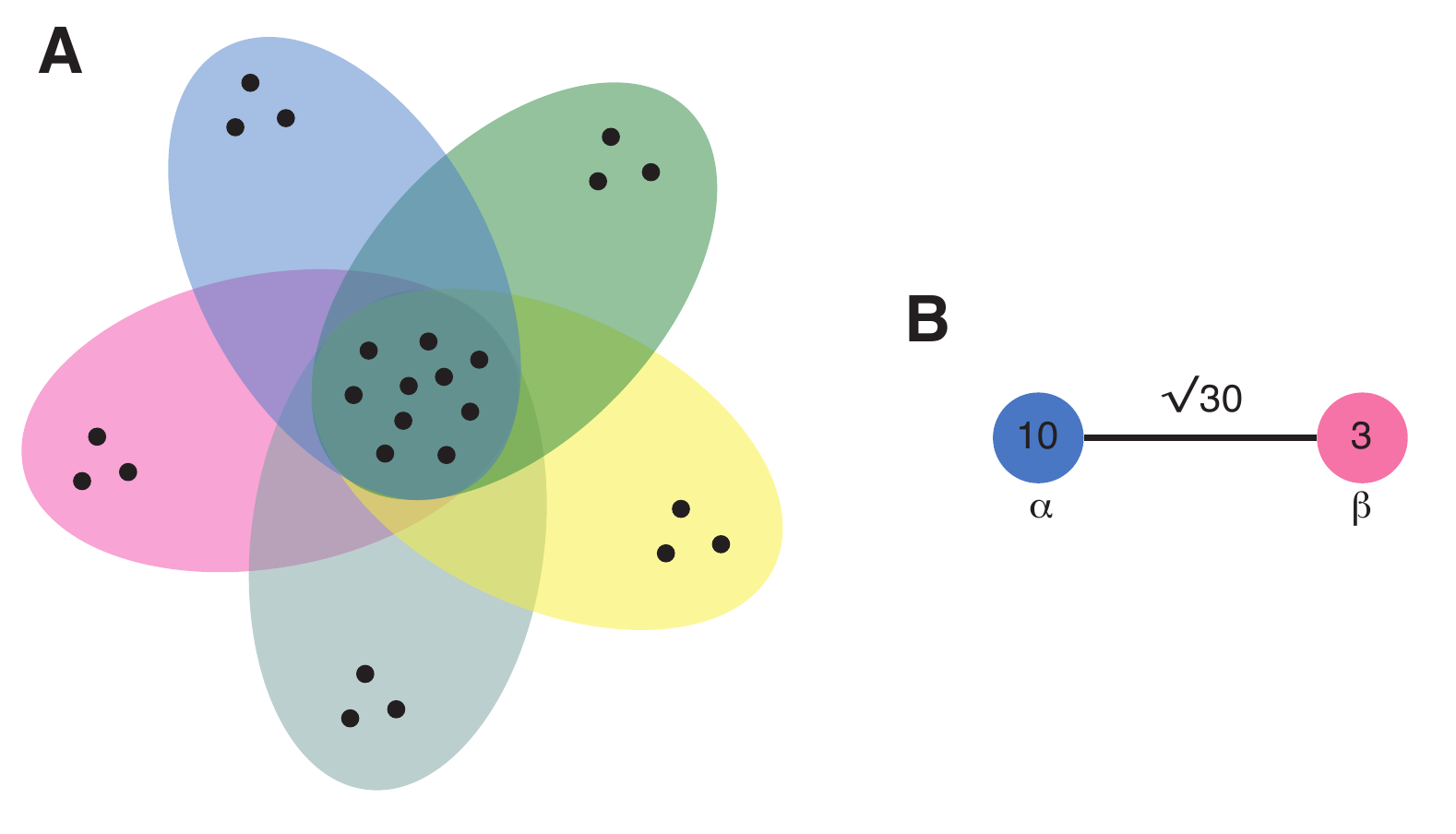}
\caption{\emph{The hyperflower}. (A) The $5$-hyperflower with $3$ twins on $25$ vertices. (B) Its quotient network. In the quotient network, $\alpha$ represents the core vertices of the hyperflower, while $\beta$ represents the peripheral vertices.}
\label{fig:quotienthyperflower}
\end{figure}

\begin{ex}[Hyperflowers]\label{ex:hyperflowers}Consider the $l$-hyperflower with $t$-twins introduced in \cite{AndreottiMulas} and shown in Figure \ref{fig:quotienthyperflower}(A). This is a hypergraph $\Gamma=(V,H)$ with only inputs whose vertex set can be written as $ V=W\sqcup V_1\sqcup \ldots\sqcup V_l$, where each $V_j$ has cardinality $l$, and the hyperedge set is given by$$H=\{h_j=W\cup V_j \mbox{ for } j=1,\ldots,l \}.$$

As shown in \cite[Lemma 6.12]{AndreottiMulas}, the spectrum of $\Gamma$ is given by:
\begin{itemize}
    \item $0$, with multiplicity $n-l$.
   \item $t$, with multiplicity $l-1$. As corresponding eigenfunctions one can choose the $f_j$'s, for $j\in\{2,\ldots,l\}$, that are $1$ on $V_1$, $-1$ on $V_j$ and $0$ otherwise.
   \item $n-tl+t$, and the constant functions are the corresponding eigenfunctions.\end{itemize}
   It's easy to see that $\Gamma$ has two orbits and, in this case, the adjacency automorphisms coincide with the Laplacian automorphisms. Thus, the redundancy of the hyperflower is $r=1/n$. Moreover, the quotient network only has two vertices $\alpha$ and $\beta$ representing the core vertices and the peripheral vertices of $\Gamma$, respectively. Its adjacency matrix is $Q^{\textrm{sym}}$, where
  \begin{align*}   Q^{\textrm{sym}}_{\alpha\beta}&=\frac{1}{\sqrt{|V_\alpha|\cdot|V_\beta|}}\cdot \sum_{i\in V_\alpha, j\in V_\beta }\left(-\frac{A_{ij}}{\sqrt{\deg(i)\deg(j)}}\right)\\
  &=\frac{1}{\sqrt{(n-tl)(tl)}}\cdot (n-tl)(tl)\cdot\frac{1}{\sqrt{l}}\\ &=\sqrt{(n-tl)t}
   \end{align*}while
   \begin{align*} Q^{\textrm{sym}}_{\alpha\alpha}&=\frac{1}{|V_\alpha|}\cdot\left( \sum_{(i,j):i\neq j\in V_\alpha}\left(-\frac{A_{ij}}{\deg (i)}\right)+\sum_{i\in V_\alpha}1\right)\\
   &=\frac{1}{|V_\alpha|}\cdot\biggl( |V_\alpha|\cdot (|V_\alpha|-1)+|V_\alpha|\biggr)\\ &=|V_\alpha|=n-tl
   \end{align*}
   and
   \begin{align*} Q^{\textrm{sym}}_{\beta\beta}&=\frac{1}{|V_\beta|}\cdot\left( \sum_{(i,j):i\neq j\in V_\beta}\left(-\frac{A_{ij}}{\deg (i)}\right)+\sum_{i\in V_\beta}1\right)\\
   &=\frac{1}{tl}\left(tl(t-1)+tl\right)=t.
   \end{align*}
   Therefore, the quotient network has edges $(\alpha,\beta)$, $(\alpha,\alpha)$ and $(\beta,\beta)$ with weights given by $\sqrt{(n-tl)t}$, $n-tl$ and $t$, respectively. \newline
   
 For the hyperflower in Figure \ref{fig:quotienthyperflower}, for instance, the edge $(\alpha,\beta)$ has weight $\sqrt{30}$, the loop $(\alpha,\alpha)$ has weight $n-tl=10$ and the loop $(\beta,\beta)$ has weight $t=3$ (Figure \ref{fig:quotienthyperflower}). Therefore
 \begin{equation*}
     Q^{\textrm{sym}}=\left(\begin{matrix}10 &\sqrt{30}\\ \sqrt{30} &3
     \end{matrix}\right).
 \end{equation*}It's easy to check that the eigenvalues of this matrix are $13$ and $0$. Therefore, in this case, Proposition \ref{prop:spectrum} tells us that:
 \begin{itemize}
     \item $0$ and $13$ are eigenvalues for the hyperflower, with eigenfunctions that are constant on the peripheral vertices and constant on the core vertices;
     \item The other eigenvalues of the hyperflower belong to eigenfunctions that sum to zero on the peripheral vertices.
 \end{itemize}
These results are clearly in accordance with the alternative calculations given above (see also \cite[Lemma 6.12]{AndreottiMulas}).% that both $13$ and $0$ are eigenvalues. We also know that the constants are eigenfunctions for $13$ (and clearly these are constant on each orbit). Furthermore, we also know from \cite[Lemma 6.12]{AndreottiMulas} that the eigenfunctions for $t$, an eigenvalue for $\Gamma$ that does not belong to the spectrum of the quotient matrix, sum to zero on each orbit.
 \end{ex}
 
\section{Signed Automorphisms}\label{section:signed}
The results presented so far straightforwardly extend the theory of automorphisms of graphs to hypergraphs. However, oriented hypergraphs have additional automorphisms induced by sign changes, that are distinct from those encountered for graphs. In this section, we define signed automorphisms, and study their effect on the hypergraph spectrum. Although signed automorphisms do not have an immediate biochemical interpretation, we include discussion of them here for mathematical completeness. 

As shown in \cite[Lemma 49]{Hypergraphs}, if we reverse the role of a vertex $v$ in all the hyperedges in which it is contained, i.e.\ if we let it become an input where it is an output and vice versa, the spectrum doesn't change, while the eigenfunctions differ by a change of sign on $v$. More generally, given an oriented hypergraph $\Gamma$ we can reverse the role of a subset of $k$ vertices $1,\ldots,k$ and obtain a hypergraph $\Gamma'$ which is isospectral to $\Gamma$. Thus, we can apply the theory of Laplacian automorphisms to $\Gamma'$ and translate the results to $\Gamma$. We formalize this idea as follows.
\begin{definition}
Let $\sigma:V=\{1,\ldots,n\}\rightarrow\{+1,-1\}$ be a sign function. Given a permutation $p$ of the vertices of $\Gamma$, we define $p^\sigma:V\rightarrow \{\pm 1,\ldots,\pm n\}$ by letting
\begin{equation*}
    p^\sigma(i):=\sigma(i)\cdot p(i)
\end{equation*}and we say that $p^\sigma$ is a \emph{signed permutation} of the vertices. 
\end{definition}
\begin{definition}
Given a sign function $\sigma:V=\{1,\ldots,n\}\rightarrow\{+1,-1\}$, we let $\sigma(\Gamma)$ be the oriented hypergraph constructed from $\Gamma$ by reversing the role of the vertices $i$ such that $\sigma(i)=-1$, in all hyperedges in which they are contained. We say that the quotient network $Q(\sigma(\Gamma))$ of $\sigma(\Gamma)$ is a \emph{signed quotient network} of $\Gamma$.
\end{definition}

Using these definitions we can now extend the theory of hypergraph automorphisms to signed automorphisms. In particular, 

\begin{definition}
A \emph{signed hypergraph automorphism} is a signed permutation $p^\sigma$ of the vertices of $\Gamma$ such that
    \begin{equation*}
    p(h)=(p(h_{in}),p(h_{out})) \in H(\sigma(\Gamma)) \quad \text{for all } h=(h_{in},h_{out}) \in H(\Gamma).
\end{equation*}
Similarly, a \emph{signed adjacency automorphism} is a signed permutation $p$ of the vertices of $\Gamma$ such that $\bigl(A(\Gamma)\bigr)_{p(i)p(j)} = \bigl(A(\sigma(\Gamma))\bigr)_{ij}$ for all $1 \le i, j \le n$ and a \emph{signed Laplacian automorphism} is a signed adjacency automorphism $p^\sigma$ that preserves degrees, that is, $\deg (i)=\deg (p(i))$, for all $i=1,\ldots,n$.
\end{definition} 
We denote by $\aut_{\textrm{signed}}(\Gamma)$ the group of 
signed Laplacian automorphisms of $\Gamma$. Moreover, 
\begin{definition}
    The \emph{signed orbit} of $i\in V$ is
    \begin{equation*}
        \mathcal{O}^\sigma(i):=\{p^\sigma(i):p^\sigma\in\aut_{\textrm{signed}}(\Gamma)\}.
    \end{equation*}
\end{definition}
In order to make functions on orbits well defined, given $f:V\rightarrow\mathbb{R}$ we let
    \begin{equation*}
        f(-i):=-f(i),\qquad\text{for }i\in V.
    \end{equation*}

Using this notation, the following proposition is the analogue of Proposition \ref{prop:duplicate-twin} for anti-twin and anti-duplicate vertices.

\begin{proposition}\label{prop:anti}
Let $\Gamma$ be an oriented hypergraph. Given $i,j\in V$, let $p$ be the transposition $p=(i,j)$ and let $\sigma$ be the sign function such that $\sigma(i)=-1$ and $\sigma(k)=+1$, for all $k\in V\setminus\{i\}$.
\begin{itemize}
    \item[(i)] If $i$ and $j$ are anti-duplicate then $p^\sigma$ is a signed adjacency automorphism.
    \item[(ii)]  If $i$ and $j$ are anti-duplicate and $\deg(i)=\deg(j)$, then $p^\sigma$ is a signed Laplacian automorphism.
        \item[(iii)] If $i$ and $j$ are anti-twin then $p^\sigma$ is a signed hypergraph automorphism.
\end{itemize}
The converses of these statements are not necessarily true.
\end{proposition}
\begin{proof}
Analogous to the proof of Proposition \ref{prop:duplicate-twin}.
\end{proof}

We may now decompose the spectrum of $\Gamma$ taking into account signed automorphisms.

\begin{proposition}\label{prop:signedquotient}
Let $\sigma:V\rightarrow\{+1,-1\}$. The spectrum of $\Gamma$ consists of the adjacency spectrum of $Q(\sigma(\Gamma))$ (with eigenfunctions that are constant on each signed orbit) together with the eigenvalues belonging to eigenfunctions that sum to zero on each signed orbit.
\end{proposition}
\begin{proof}
By \cite[Lemma 49]{Hypergraphs}, it easily follows that $\lambda$ is an eigenvalue for $\sigma(\Gamma)$ with eigenfunction $f$ if and only if $\lambda$ is an eigenvalue for $\Gamma$ with eigenfunction $\sigma\cdot f$, where $\sigma f(i):=\sigma(i)\cdot f(i)$. Together with Corollary \ref{cor:quotient}, this proves the claim.
\end{proof}

To illustrate these ideas we again consider an example. 
\begin{ex}[Signed Hyperflower]\label{ex:signedhyperflower}
For the hyperflower in Example \ref{ex:hyperflowers} all vertices are inputs. If we let one vertex $v$ become an output in all hyperedges in which it is contained, then the theory of (unsigned) Laplacian automorphisms cannot detect this reversal. However, by choosing the sign function $\sigma:V\rightarrow\{+1,-1\}$ that has value $-1$ on $v$ and value $+1$ otherwise, and applying Proposition \ref{prop:signedquotient} its effect can be detected.
\end{ex}

These results finally give us an alternative notion of redundancy.
\begin{definition}
The \emph{signed redundancy} is
\begin{equation*}
    r_{\textrm{signed}}:=\min_{\sigma:V\rightarrow\{+1,-1\}}\frac{\#\mathcal{O}^\sigma-1}{n}
\end{equation*}
\end{definition}
By choosing $\sigma:V\rightarrow\{+1,-1\}$ with $+1$ on all vertices, we have $\mathcal{O}^\sigma(i)=\mathcal{O}(i)$ for each $i\in V$, and therefore 
\begin{equation*}
    r_{\textrm{signed}}=\min_{\sigma:V\rightarrow\{+1,-1\}}\frac{\#\mathcal{O}^\sigma-1}{n}\leq \frac{\#\mathcal{O}-1}{n}=r.
\end{equation*}
Hence, the signed redundancy is more precise than the unsigned redundancy. In the case of Example \ref{ex:signedhyperflower}, for instance, $r_{\textrm{signed}}=1/n$ while $r=2/n$.

\section{Example: Basic Enzyme Reactions}
In order to illustrate this theory, consider the basic enzyme reactions:
\begin{equation}\label{eq:backreaction}
    E+S \stackrel[k_f]{k_r}{\leftrightarrows} ES \xrightarrow{k_{cat}} E+P,
\end{equation}
where $k_f$, $k_r$ and $k_{cat}$ are reaction rates. To explore the geometry of this system we will consider two hypergraph models in which the above chemical substances are represented by vertices, and the reactions are represented by hyperedges. The first hypergraph model accounts for forward reactions only, and so represents the system: 
\begin{equation}\label{eq:systemenz}
    E+S \xrightarrow{k_f} ES \xrightarrow{k_{cat}} E+P.
\end{equation}
We let $\Gamma:=(V,H)$, be a hypergraph, where the vertex set is $V:=\{E,S,ES,P\}$, the hyperedge set is $H:=\{h_1,h_2\}$, and the oriented hyperedges are $h_1:=(\{E,S\},\{ES\})$ and $h_2:=(\{ES\},\{E,P\})$. This hypergraph is illustrated in Fig. \ref{fig:enzyme}.

    \begin{figure}[t!]
        \centering
        \includegraphics[width=7cm]{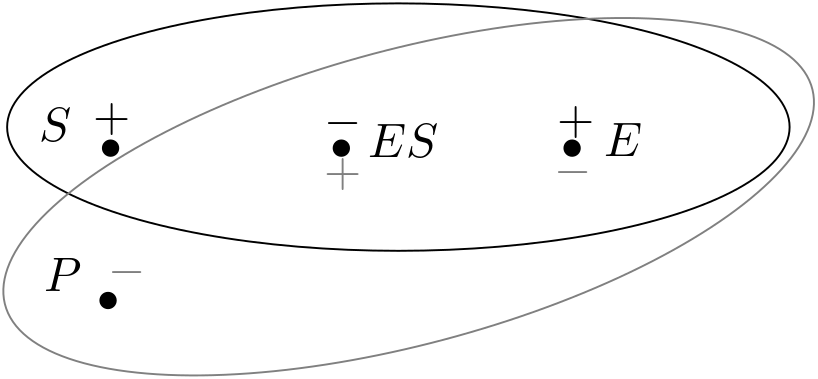}\caption{The hypergraph representing the system given in Eq.~\eqref{eq:systemenz}. }\label{fig:enzyme}
    \end{figure}

The spectrum of $\Gamma$, which is a  $2$-hyperflower with $1$ twin on $4$ vertices, is $0, 0, 1, 3$. In this case, there are exactly two non-zero eigenvalues because there are two hyperedges and these hyperedges are independent of each other (cf.\ \cite{Hypergraphs}). The largest eigenvalue is $3$ because the hypergraph is \emph{bipartite} and each reaction contains exactly three substances (cf.\ \cite{Sharp}). Finally, $1$ is an eigenvalue because the vertices $S$ and $P$ are anti-duplicate. A corresponding eigenfunction is $f:V\rightarrow\mathbb{R}$ such that $f(S)=f(P)=1$ and $f(ES)=f(S)=0$. Since there are no hypergraph automorphisms, the redundancy is 
\begin{equation*}
        r=\frac{\#\mathcal{O}-1}{4}=\frac{3}{4}.
\end{equation*}
However, because $S$ and $P$ are anti-duplicate, and $E$ and $ES$ are anti-twin the system possess \emph{signed} automorphisms. These symmetries are not present in graph representations of this system, and so represent features of the chemical reaction system that are specifically identified by the hypergraph theory. Thus, the signed redundancy differs from the redundancy. In this case, the signed redundancy is
    \begin{equation*}
    r_{\textrm{signed}}=\min_{\sigma:V\rightarrow\{+1,-1\}}\frac{\#\mathcal{O}^\sigma-1}{n}=\frac{1}{4},
    \end{equation*}
and the minimum is achieved for the function $\sigma:V\rightarrow\{+1,-1\}$ that has value $1$ on $S$, $ES$ and value $-1$ on $E$, $P$. Its signed orbits are $\{S,P\}$ and $\{E,ES\}$. It should be noted that these orbits are coincident with the conservation laws of the dynamics, but this is not always the case. Conservation laws do not relate directly to automorphisms or signed automorphisms, but rather are related to properties of another Laplacian, as discussed in \cite{Hypergraphs}. 

These results demonstrate, via a practical empirical (rather than theoretical) example, that there are geometric properties that are detected by the signed automorphisms and are not detected by the automorphisms. However, this example only accounts for forward reactions. In order to take the backward reaction in the system described by Eq.\ \eqref{eq:backreaction} into account, we consider the hypergraph $\Gamma\cup h_3:=(V,H\cup h_3)$, where $h_3:=(\{ES\},\{E,S\})$. The eigenvalues of $\Gamma\cup h_3$ coincide with the eigenvalues of $\Gamma$, counted with multiplicity, but the eigenfunctions and the redundancy change.

The hypergraph $\Gamma\cup h_3$ has two non-zero eigenvalues (as does $\Gamma$). In this case, although there are three hyperedges (i.e., reactions), only two of them are independent, since $h_3$ and $h_1$ are inverse of one another (cf.\ \cite{Hypergraphs}). Moreover, as in the case of $\Gamma$, the largest eigenvalue is $3$ because the hypergraph is \emph{bipartite} and each reaction involves three substances (cf.\ \cite{Sharp}). Finally, although $P$ and $S$ are not anti-duplicate in this case, $\Gamma\cup h_3$ is isospectral with $\Gamma$ and so $\Gamma\cup h_3$ inherits the eigenvalue $1$ due to the fact that $S$ and $P$ are anti-duplicate in $\Gamma$. This endows the $\Gamma\cup h_3$ with a \emph{shadow} symmetry. As with $\Gamma$, there are no automorphisms, and so 
       \begin{equation*}
        r=\frac{\#\mathcal{O}-1}{4}=\frac{3}{4}.
    \end{equation*}
    
    However, in this case, the signed redundancy is
     \begin{equation*}
    r_{\textrm{signed}}=\min_{\sigma:V\rightarrow\{+1,-1\}}\frac{\#\mathcal{O}^\sigma-1}{n}=\frac{1}{2},
    \end{equation*}
    and this minimum is achieved for $\sigma:V\rightarrow\{+1,-1\}$ that has value $1$ on $S$, $ES$, $P$ and value $-1$ on $E$. Its signed orbits are $\{S\}$, $\{P\}$ and $\{E,ES\}$. The difference arises because $S$ and $P$ are now not anti-duplicate, hence they do not belong to a same signed orbit. 
    
    This second example has shown that, while adding reversed hyperedges (reactions) does not change the hypergraph spectrum, the eigenfunctions, signed automorphisms and signed redundancy can change. Moreover, certain spectral properties of the hypergraph that includes reversed hyperedges are derived from the structure of the simpler hypergraph without them. Thus, there is a motivation for studying simplified systems without losing structural information. 

\section*{Discussion}
Biochemical reaction systems often contain duplication which manifests as symmetry in their underlying hypergraphs. Here, we have introduced and studied automorphisms for oriented hypergraphs. We focused on the normalised Laplacian, which is known to encode many qualitative properties of a hypergraph, and have generalized the known theory for graphs \cite{Chung,Brouwer}. We have shown that, while the generalisation to the case of classical hypergraphs is intuitive and relatively straightforward, for a complete theory in the case of \emph{oriented} hypergraphs, one needs additional constructions, such as the \emph{signed} automorphisms and \emph{signed} redundancy. Thus, the general theory we have introduced extends that of graphs and hypergraphs to the more general -- and appropriate for modelling complicated biochemical reaction systems inside a cell -- case of oriented hypergraphs. To illustrate this theory we have shown, with a simple practical example, that it can be used to practically study redundancy in biochemical systems. 

There has been some prior work on spectral graph theory applied to biochemical networks, see for instance \cite{symm1,symm2,symm3,bio2006,bio2007,bio2009,bio2009JJ,bio2019}, and there is a growing literature on how to use hypergraphs for modelling biochemical networks. In \cite{Estrada2006}, for example, the concepts of subgraph centrality and clustering are generalised to the case of hypergraphs, and various practical examples, including examples from biology, are given. Similarly, in \cite{curvature}, a notion of curvature for hypergraphs is proposed and applied to the analysis of the E.\ coli metabolism. In \cite{Cellularnetworks}, it is argued -- using a range of practical examples -- that biological networks that are typically modelled as graphs can also be fruitfully modelled using hypergraphs. Some practical algorithms that do not involve spectral theory, as well as network statistics for hypergraphs, are also discussed. In \cite{Stadler2015}, some mathematical foundations (which again do not include spectral theory) for the study of hypergraphs in the context of chemical reaction systems and biological evolution are given. Similarly, in \cite{2014signaling,2019signaling}, hypergraphs are used as a model for signaling pathways in cellular biology. In \cite{2014signaling}, in particular, it is noted that, since hypergraph theory is less well-known than graph theory, there is a need to develop theoretical and algorithmic foundations for hypergraphs.

However, although there is a growing literature on both spectral graph theory applied to biology and hypergraph modelling of biochemical networks, we are still lacking theoretical tools needed to apply spectral hypergraph theory to biochemical networks. In this paper, we have taken a step further in this direction. 

In the future it will be interesting to analyze large, complex biochemical networks using spectral hypergraph methods. There are a number of publicly available, curated, repositories of biochemical reaction systems \cite{rep1,rep2,rep3,Link,KEGG}. Once in a hypergraph format, spectral properties of such empirical networks can then be determined by considering their associated matrices such as the normalised Laplacian, which we have focused on here. As noted, the eigenvalues of these matrices encode many important qualitative properties of the underlying hypergraph \cite{Hypergraphs,Master-Stability,Sharp,MulasZhang,AndreottiMulas,spectralclasses,hyp2014,hyp2015,hyp2019,chen2018,ReffRusnak}, including its symmetries and associated redundancy, and these properties, in turn, shed light on the essential structural properties of the system under study. By converting a geometric problem into an algebraic one the benefits of this approach are numerous, since they make the structure of the system amenable to detailed analysis. These benefits include computational aspects, since the spectrum of a square matrix can be computed with relatively little computational effort.

Moreover, the tools presented here may also be useful in the
analysis of chemical reaction networks more generally -- particularly in applications that involve complex sets of reactions, for example as encountered in some industrial processes, where redundancy may also be ubiquitous.

\begin{acknowledgements}
 We thank the anonymous referees for valuable suggestions and comments. 
\end{acknowledgements}

\section{Conflict of interest}
The authors declare that they have no conflict of interest.

%\bibliographystyle{spmpsci}       %mathematics and physical sciences
\bibliographystyle{spphys}       % APS-like style for physics
\bibliography{Geometry}   % name your BibTeX data base

\begin{thebibliography}{10}
\providecommand{\url}[1]{{#1}}
\providecommand{\urlprefix}{URL }
\expandafter\ifx\csname urlstyle\endcsname\relax
  \providecommand{\doi}[1]{DOI \discretionary{}{}{}#1}\else
  \providecommand{\doi}{DOI \discretionary{}{}{}\begingroup
  \urlstyle{rm}\Url}\fi

\bibitem{newman2018networks}
M.~Newman, \emph{Networks} (Oxford university press, 2018)

\bibitem{barabasi2016network}
A.L. Barab{\'a}si, et~al., \emph{Network science} (Cambridge university press,
  2016)

\bibitem{newman2003structure}
M.E. Newman, SIAM review \textbf{45}(2), 167 (2003)

\bibitem{barabasi1999emergence}
A.~Barab{\'a}si, R.~Albert, Science \textbf{286}(5439), 509 (1999)

\bibitem{milo2002network}
R.~Milo, S.~Shen-Orr, S.~Itzkovitz, N.~Kashtan, D.~Chklovskii, U.~Alon, Science
  \textbf{298}(5594), 824 (2002)

\bibitem{watts1998collective}
D.~Watts, S.~Strogatz, Nature \textbf{393}(6684), 440 (1998)

\bibitem{symm1}
B.~MacArthur, R.~Sánchez-García, J.~Anderson, Discrete Applied Mathematics
  \textbf{156}(18), 3525 (2008)

\bibitem{symm3}
R.~Sánchez-García, Commun. Phys. \textbf{3}(87) (2020)

\bibitem{pecora2014cluster}
L.~Pecora, F.~Sorrentino, A.~Hagerstrom, T.~Murphy, R.~Roy, Nature
  communications \textbf{5}(1), 1 (2014)

\bibitem{klickstein2019symmetry}
I.~Klickstein, L.~Pecora, F.~Sorrentino, Chaos: An Interdisciplinary Journal of
  Nonlinear Science \textbf{29}(7), 073101 (2019)

\bibitem{wu2010k}
W.~Wu, Y.~Xiao, W.~Wang, Z.~He, Z.~Wang, in \emph{Proceedings of the 13th
  international conference on extending database technology} (2010), pp.
  111--122

\bibitem{Chung}
F.~Chung, American Mathematical Society  (1997)

\bibitem{Brouwer}
A.~Brouwer, W.~Haemers, \emph{Spectra of Graphs} (Springer, New York, NY, 2012)

\bibitem{carlsson2009topology}
G.~Carlsson, Bulletin of the American Mathematical Society \textbf{46}(2), 255
  (2009)

\bibitem{zomorodian2005topology}
A.~Zomorodian, \emph{Topology for computing}, vol.~16 (Cambridge university
  press, 2005)

\bibitem{Hypergraphs}
J.~Jost, R.~Mulas, Advances in Mathematics \textbf{351}, 870 (2019)

\bibitem{Cellularnetworks}
S.~Klamt, U.~Haus, F.~Theis, PLoS Comput Biol. \textbf{5}(5) (2009).
\newblock 2009 May; 5(5): e1000385.

\bibitem{horak2013}
D.~Horak, J.~Jost, Advances in Mathematics \textbf{244}, 303 (2013)

\bibitem{shi1992}
C.J. Shi, Microelectronics journal \textbf{23}(7), 533 (1992)

\bibitem{ReffRusnak}
N.~Reff, L.~Rusnak, Linear Algebra and its Applications \textbf{437}, 2262
  (2012)

\bibitem{Master-Stability}
R.~Mulas, C.~Kuehn, J.~Jost, Phys. Rev. E \textbf{101}, 062313 (2020).
\newblock \doi{10.1103/PhysRevE.101.062313}.
\newblock \urlprefix\url{https://link.aps.org/doi/10.1103/PhysRevE.101.062313}

\bibitem{Sharp}
R.~Mulas, Mathematical notes \textbf{109}, 102 (2021)

\bibitem{MulasZhang}
R.~Mulas, D.~Zhang, Discrete Mathematics \textbf{344}, 112372 (2021).
\newblock \doi{10.1016/j.disc.2021.112372}

\bibitem{AndreottiMulas}
E.~Andreotti, R.~Mulas, {Signless Normalized Laplacian for Hypergraphs.}
\newblock ArXiv:2005.14484

\bibitem{spectralclasses}
R.~Mulas, The Australasian Journal of Combinatorics \textbf{79}, 495 (2021)

\bibitem{hyp2014}
N.~Reff, Electronic Journal of Linear Algebra \textbf{27} (2014)

\bibitem{hyp2015}
V.~Chen, A.~Rao, L.~Rusnak, A.~Yang, Linear Algebra and Its Applications
  \textbf{485}, 442 (2015)

\bibitem{hyp2019}
L.~Duttweiler, N.~Reff, Linear Algebra and its Applications \textbf{578}, 251
  (2019)

\bibitem{chen2018}
G.~Chen, V.~Liu, E.~Robinson, L.~Rusnak, K.~Wang, Linear Algebra and its
  Applications \textbf{556}, 323 (2018)

\bibitem{bretto2013hypergraph}
A.~Bretto, An introduction. Mathematical Engineering. Cham: Springer  (2013)

\bibitem{symm2}
B.~MacArthur, R.~Sánchez-García, Physical Review E \textbf{80}(2) (2009)

\bibitem{vazquez2003}
A.~V{\'a}zquez, A.~Flammini, A.~Maritan, A.~Vespignani, Complexus
  \textbf{1}(1), 38 (2003)

\bibitem{bio2006}
A.~Lesne, Letters in Mathematical Physics \textbf{78}(3), 235 (2006)

\bibitem{bio2007}
O.~Mason, M.~Verwoerd, IET systems biology \textbf{1}(2), 89 (2007)

\bibitem{bio2009}
A.~Perkins, M.~Langston, in \emph{BMC bioinformatics}, vol.~10 (BioMed Central,
  2009), vol.~10, pp. 1--11

\bibitem{bio2009JJ}
A.~Banerjee, J.~Jost, Discrete Applied Mathematics \textbf{157}(10), 2425
  (2009)

\bibitem{bio2019}
C.H. Huang, J.~Tsai, N.~Kurubanjerdjit, K.L. Ng, bioRxiv p. 536318 (2019)

\bibitem{Estrada2006}
E.~Estrada, J.~Rodríguez-Velázquez, Physica A: Statistical Mechanics and its
  Applications \textbf{364}, 581 (2006).
\newblock \doi{https://doi.org/10.1016/j.physa.2005.12.002}.
\newblock
  \urlprefix\url{https://www.sciencedirect.com/science/article/pii/S0378437105012550}

\bibitem{curvature}
W.~Leal, G.~Restrepo, P.~Stadler, J.~Jost, {Forman-Ricci Curvature for
  Hypergraphs}.
\newblock ArXiv:1811.07825

\bibitem{Stadler2015}
C.~Flamm, B.~Stadler, P.~Stadler, in \emph{Advances in Mathematical Chemistry
  and Applications}, ed. by S.~Basak, G.~Restrepo, J.~Villaveces (Bentham
  Science Publishers, 2015), pp. 300--328.
\newblock \doi{https://doi.org/10.1016/B978-1-68108-053-6.50013-2}.
\newblock
  \urlprefix\url{https://www.sciencedirect.com/science/article/pii/B9781681080536500132}

\bibitem{2014signaling}
A.~Ritz, A.~Tegge, H.~Kim, C.~Poirel, T.~Murali, Trends in biotechnology
  \textbf{32}(7), 356 (2014)

\bibitem{2019signaling}
M.~Schwob, J.~Zhan, A.~Dempsey, IEEE/ACM transactions on computational biology
  and bioinformatics  (2019)

\bibitem{rep1}
G.~Bader, I.~Donaldson, C.~Wolting, B.~Ouellette, T.~Pawson, C.~Hogue, Nucleic
  Acids Research \textbf{29}(1), 242 (2001).
\newblock \doi{10.1093/nar/29.1.242}

\bibitem{rep2}
G.~Bader, D.~Betel, C.~Hogue, Nucleic Acids Research \textbf{31}(1), 248
  (2003).
\newblock \doi{10.1093/nar/gkg056}

\bibitem{rep3}
D.~Szklarczyk, J.~Morris, H.~Cook, M.~Kuhn, S.~Wyder, M.~Simonovic, A.~Santos,
  N.~Doncheva, A.~Roth, P.~Bork, L.~Jensen, C.~von Mering, Nucleic Acids
  Research \textbf{45}(D1), D362 (2016).
\newblock \doi{10.1093/nar/gkw937}

\bibitem{Link}
S.B.R. Group (2021).
\newblock \urlprefix\url{https://systemsbiology.ucsd.edu/Downloads}

\bibitem{KEGG}
M.~Kanehisa, S.~Goto, S.~Kawashima, A.~Nakaya, Nucleic acids research
  \textbf{30}(1), 42 (2002)

\end{thebibliography}

\end{document}